\documentclass[12pt]{article}

\usepackage[dvips]{graphicx}

\usepackage[bookmarks,colorlinks]{hyperref}

\usepackage{amsmath,amstext,amssymb,amsopn,amsthm}
\usepackage{amsmath,amssymb,amsthm}
\usepackage[mathscr]{eucal}

\usepackage{graphicx}
\usepackage{amssymb}
\usepackage{amsmath}
\usepackage{color}
\usepackage{times}
\usepackage{ulem}

\newcommand{\F}{\mathcal{F}}

\newcommand{\R}{\mathbb{R}}
\newcommand{\N}{\mathbb{N}}
\newcommand{\No}{\mathbb{N}_0}
\newcommand{\E}{\mathbb{E}}
\newcommand{\PP}{\mathbb{P}}
\newcommand{\Pn}{\mathcal{P}_n}

\newcommand{\dY}{\Delta Y}
\newcommand{\dX}{\Delta X}
\newcommand{\X}{{\cal X}}
\newcommand{\XX}{{\mathbb X}}
\newcommand{\XXd}{\XX_{diag}}
\newcommand{\e}{\epsilon}

\newcommand{\Rd}{{\R^d}}
\newcommand{\Rb}{{\bar{\R}}}
\newcommand{\Rdz}{{\R^d_0}}

\newcommand{\Rp}{{\R_+}}

\newcommand{\w}{{\omega}}
\newcommand{\Prawd}{\mathbb{P}}

\newcommand{\1}{\mathbf{1}}

 \newtheorem{theorem}{Theorem}[section]

\newtheorem{lemma}[theorem]{Lemma}
\newtheorem{corollary}[theorem]{Corollary}

\newtheorem{remark}[theorem]{Remark}

\definecolor{kb}{rgb}{0.2,0.7,0.1}

\definecolor{jr}{rgb}{0.1,0.3,0.6}

\numberwithin{equation}{section}

\title{L\'evy systems and moment formulas for 
mixed Poisson integrals
 \footnote{K.~Bogdan and {\L}.~Wojciechowski 
were supported in part by NCN grant 2012/07/B/ST1/03356. 
J.~Rosi\'nski's research was partially supported by a grant \# 281440 from the Simons Foundation.}
}
\author{Krzysztof Bogdan\footnote{Faculty of Pure and Applied Mathematics, Wroc\l{}aw University of Science and Technology, 50-370 Wroc\l{}aw, Poland, e-mail: bogdan@pwr.edu.pl} \and   Jan Rosi\'nski\footnote{Department of Mathematics, University of Tennessee, Knoxville TN 37996, USA, e-mail: rosinski@math.utk.edu} 
 \and Grzegorz Serafin\footnote{Faculty of Pure and Applied Mathematics, Wroc\l{}aw University of Science and Technology, 50-370 Wroc\l{}aw, Poland, e-mail: grzegorz.serafin@pwr.edu.pl} 
 \and {\L}ukasz Wojciechowski
}

\date{\today}

\begin{document}

\maketitle

\begin{abstract}
We propose Mecke-Palm formula for multiple integrals with respect to the Poisson random measure 
and its intensity measure performed, or mixed, in an arbitrary order.
We apply the formulas to mixed L\'evy systems of L\'evy processes and obtain moment formulas for mixed Poisson integrals. 
\end{abstract}
\section{Introduction}

The Mecke-Palm formula is an important identity in sto\-cha\-stic analysis of Poisson random measures.
In this work we propose its generalization named the (multiple) mixed-type Mecke-Palm formula. We show that the generalization
is useful and has a considerable scope of applications.

Part of our motivation comes from recent results on moments of stochastic integrals \cite{MR3021456}, \cite{MR3210036}.
These were
obtained  for $1$-processes in \cite{MR3021456}
by  using combinatorics of the binomial convolution to undo the usual compensation in stochastic integration against Poisson random measures \cite{MR1876169},  \cite{MR1943877}; and they were extended  in \cite[Theorem 3.1]{MR3210036} to ensembles of integrals of $1$-processes.

By compensation in the previous paragraph we mean integration against the difference of the random Poisson measure and its intensity, or control, measure. It is well-known that such integration fits well into the framework of $L^2$ Hilbert spaces \cite{MR2791919}. In opposition, the results of this paper mainly concern iterated integrations against the (uncompensated) Poisson random measure interlaced, or mixed, with integrations against the control measure. Such integrations preserve nonnegativity and are performed under nonnegativity or absolute integrability conditions, rather then the square-integrability conditions (for which see Lemma~\ref{martingaleLevy} below or \cite{MR2791919}).
In both settings, however, the main feature of the iterated stochastic integration is the impact of the diagonals in the corresponding Cartesian products of the state space, which cannot be ignored because the random measure has atoms. The impact is accounted for by using partitions of the set of coordinates.  We shall see below that in the setting of the uncompensated stochastic integration the description is  simpler than in
the compensated, or $L^2$, setting, for which we refer the reader to \cite[Chapter~5]{MR2791919}. In fact, the integrals against the compensated Poisson measure can be considered as (limits of) linear combinations of mixed integrals with respect to the Poisson random measure and its control measure, which explains the added complexity. 
Moreover, we may consider the results obtained in both settings as consequences of the mixed Mecke-Palm formula and the structure of the family of partitions. 
Our presentation is essentially self-contained in that it relies on the mixed Mecke-Palm formula, which we explain from the first principles. We should also remark that the integrands we consider are random, and in this respect they are more general than those in \cite{MR2791919}. 
A complete survey of results on integration with respect to random measures is beyond the scope of this paper, 
but for more information we like to refer the reader to \cite{MR1012497, MR1167198, MR874529}.

Below we first prove the mixed Mecke-Palm formula and use its along with the so-called linearization to obtain moments of stochastic integrals in more generality than known before:
we consider moments of $k$-processes with arbitrary integer $k\ge 1$, and we allow
 Poisson stochastic integrations to be interlaced, or mixed, up to arbitrary multiplicity and order, with integrations against the intensity measure of the Poisson random measure.
Our proofs are more direct as compared to \cite{MR3021456} and \cite{MR3210036}, because they easily follow
from the mixed-type Mecke-Palm formula. 

When  the random measure is given by the jumps of a L\'evy process, the mixed Mecke-Palm formula 
translates into multiple L\'evy systems of mixed type, which is our second main application. By the multiple L\'evy systems of mixed type we mean identities for expectations of functionals defined by accumulated summations indexed by the jumps of the L\'evy process and integrations against the product of the linear Lebesgue measure on the time scale and the L\'evy measure of the process in space. They generalize 
the classical (single) L\'evy system \cite{MR0343375, MR745449, MR1138461},
which is an important tool in the study of jump-type Markov processes.
The multiple variants have interesting applications
and we indicate some of them.

The structure of the paper is as follows. In Theorem~\ref{lem:multiPalmDiag} of Section~\ref{sec:iMPf} we give the mixed Mecke-Palm formula for $k$-processes.
In Theorem~\ref{main:Theorem} of Section~\ref{sec:a} we derive general moment formulas for ensembles of $k$-processes. These are illustrated by the moments formulas for $1$-processes and $2$-processes in Section~\ref{sec:mf1} and Section~\ref{sec:mf2}, respectively.
In Theorem~\ref{E} of Section~\ref{s:Ls} we present the multiple mixed L\'evy systems for L\'evy processes in $\Rd$.
In Section~\ref{sec:asLs} 
we present several applications of the L\'evy systems including applications that merge the topics and techniques from 
Section~\ref{sec:a} and Section~\ref{s:Ls}.
Some of the results are known, but even then the presentation may be of interest. 
In Section~\ref{sec:MP} we give a proof of the simple Mecke-Palm formula, to make the paper more self-contained.

{\bf Acknowledgments.} The inspiration to study L\'evy systems came from the joint work of the first named author with Rodrigo Ba\~nuelos \cite{MR2345912}, and our interest in moment formulas is due to the work of Nicolas Privault \cite{MR3021456}. 
The present paper is based in part on the PhD dissertation of the fourth named author \cite{dissert}.  
We thank Mateusz Kwa\'snicki for many discussions and suggestions, Jean Jacod and Alexandre Popier for references to literature, Anita Behme for discussions on the L\'evy integral in Section~\ref{sec:asLs}, Nicolas Privault for a discussion of our results and Aleksander Janicki and Rodrigo Ba\~nuelos for the inspirations.

\section{Mixed Mecke-Palm formulas}\label{sec:iMPf}
A direct approach to calculus of Poisson random measures is based on the configuration space: 
Given a locally compact separable metric space  $\XX$, any locally finite subset of $\XX$ is called a configuration on $\XX$.
The configuration space is defined as $\Omega=\Omega_{\XX}=\{\w \subset \XX: \w \text{ is a configuration on $\XX$} \}$  \cite{MR2531026}.
 The elements of $\Omega$ can be identified with the class of locally finite, nonnegative-integer valued measures: if $\omega=\{y_1, y_2,\ldots\}$, where $y_i\in
\XX
$ are all different, then we also write
\begin{align*}
\w = \sum_i \delta_{y_i},
\end{align*}
where $\delta_y$ is the probability measure concentrated at $y\in \XX$. According to this identification,  $\omega$ will have two meanings depending on the context: a configuration on $\XX$ or a measure on $\XX$. 
We equip $\Omega$ with a $\sigma$-algebra $\mathcal{F}$, which is the smallest sigma-algebra of subsets of $\Omega$ making the maps $\omega\mapsto \omega(A)$ measurable for each Borel set $A\subset \XX$ cf. \cite[Chapter~10]{MR1876169}. 
A jointly measurable map 
$$f: (\XX)^k\times \Omega \ni (x_1,\ldots,x_k;\w)\;\mapsto \;f(x_1,\ldots,x_k;\w) \in \Rb$$ 
is called a process or, more specifically, a $k$-process. Here $\Rb=\R\cup\{-\infty,\infty\}$, $k\in \No=\{0,1,\ldots\}$, and when $k=0$, i.e., 
$f:\;\Omega \ni \w\mapsto f(\omega)\in \Rb$,  we call $f$ a random variable. 
We also note that for every Borel function $\phi\ge 0$ on $\XX$, the map
$$\omega \mapsto \int \phi(x)\omega(dx)$$
is well-defined and measurable, hence a random variable. We say that a $k$-process $f$ depends only on $\X \subseteq \XX$, if $f(x_1,\ldots,x_k; \w) = f(x_1,\ldots,x_k ; \w \cap \X)$ for all 
$\w \in \Omega$ and $x_1,\ldots,x_k \in \XX$.

We let $\XXd^n = \{ x = (x_1, \cdots, x_n) \in \XX^n: x_i = x_j$ \mbox{for some} $i \neq j \}$ and 
$\XX_{\neq}^n = \XX^n \setminus \XXd^n$, where $\XX_{\neq}^1=\XX$.
Given a $k$-process $f$
and $n\in \N$ we define the $n$-th coefficient $f_{(n)}$ of $f$ as a function $f_{(n)}: \XX^k\times \XX_{\neq}^n \mapsto \Rb$ such that
\begin{equation*}
f_{(n)}(x_1,\ldots,x_k;y_1, \ldots, y_n)=f(x_1,\ldots,x_k;\w), \quad \text{where }\;  \w = \{y_1, \ldots, y_n\}.
\end{equation*} 
We also let $f_{(0)}(x_1,\ldots,x_k)=f(x_1,\ldots,x_k; \emptyset)$.
Thus,
coefficients $f_{(n)}$ are 
Borel functions on $\XX^k\times \XX_{\neq}^n$
invariant upon permutations of the last $n$ coordinates.
In particular, for random variables ($0$-processes) $f$ we simply have $f_{(n)}(y_1, \ldots, y_n)= f(\{y_1, \ldots, y_n\})$, where $y_1, \ldots, y_n$ are all different, and $f_{(0)}=f(\emptyset)$.
Of course, if $f$  is a $k$-process, then $\w\mapsto f(x_1,\ldots,x_k;\w)$ is a random variable for every choice of $x_1,\ldots,x_k\in \XX$, and the $n$-th coefficient of this random variable is $f_{(n)}(x_1,\ldots,x_k;y_1,\ldots,y_n)$, provided $(y_1, \ldots, y_n)\in \XX_{\neq}^n$.

Now we define a Poisson probability measure $\PP$ on $(\Omega, \mathcal{F})$ and the corresponding expectation $\E$.
Notice that $N(A, \omega) := \omega(A)$ is a random measure on $\XX$ under any probability measure on $\Omega$, but we will consider the probability $\PP$ which makes $N$ a Poisson random measure with intensity measure $\sigma(A) = \E N(A)$. Here is a construction of $\PP$.
The main analytic datum  is a non-atomic measure $\sigma$ finite on compact subsets of $\XX$. 
If $\X$ is a Borel subset of $\XX$
and $\sigma(\X)<\infty$,
then the corresponding probability, say $\PP_{\X}$, is concentrated on {\it finite} configurations $\Omega_{\X}$ on $\X$ and defined by
\begin{align}
\E_{\X} f&=\int_{\Omega_{\X}} f(\w)\PP_{\X}(d\w)\nonumber\\
&= e^{-\sigma(\X)}\sum_{n=0}^{\infty}\frac{1}{n!} \int_{\X^n}f_{(n)}(y_1, \ldots, y_n) \sigma(dy_n)\cdots \sigma(dy_1),\label{emSie}
\end{align}
cf. 
\cite[p. 196]{MR2531026}.
Here  the first term on the rightmost of (\ref{emSie}) is $e^{-\sigma(\X)}f_{(0)}$,
according to 
a
general convention.

Further, 
let Borel sets $\X_1, \X_2, \ldots \subset \XX$ be such that $\bigcup_m \X_m = \XX$, $ \X_m \cap \X_n = \emptyset$ for $ m \neq n$,  and $\sigma(\X_m)<\infty$ for every $m$. We identify
$\Omega_\XX$ with $\otimes_m \Omega_{\X_m}$ by identifying $\omega$ with $(\omega\cap \X_m)_m$.
Then $\PP$ is unambiguously defined 
as the product measure,
$$\PP =  \otimes_{m} \PP_{\X_m}.$$ 
For $\X\subset \XX$, $\PP_\X$ may be considered as a marginal distribution of $\PP$, and
for 
random variables $f_1$, $f_2$ depending only on disjoint 
$\X_1,\X_2\subset \XX$, respectively, we have
\begin{align}\label{eq:ind}
\E [f_1(\w)f_2(\w)]
=\E_{\X_1} [f_1(\w)]\ \E_{\X_2} [f_2(\w)].
\end{align}
Here the notions of independence of a function from a set of arguments, and the pro\-babilistic independence happily meet.
In what follows $\E$ and $\PP$ are always the expectation and distribution making $\omega$ a Poisson random measure with control measure $\sigma$ (in Section~\ref{s:Ls} we make additional structure assumptions on $\XX$ and $\sigma$).

In what follows we denote $\w_{1}=\omega$,  $\w_{0} = \sigma$, for $\w \in \Omega$.
For a $1$-process $f\ge 0$
and $\epsilon\in \{ 0 , 1 \}$, we have
\begin{equation}\label{Palm2} 
   \E \int_{\XX} f(x; \omega) ~\w_\epsilon(dx) =  \int_{\XX} \E f\left(x;
\omega + \epsilon \delta_{x} \right)~ \sigma(dx).
\end{equation}
Indeed, for $\epsilon=0$ the identity follows from Fubini-Tonelli, and if $\epsilon=1$, then it is the celebrated Mecke-Palm formula, see also \cite[(2.10)]{Last_Penrose}.
(For the readers's convenience a direct proof of the Mecke-Palm formula is given in Section~\ref{sec:MP}.)

We say that a $k$-process $f$ 
vanishes on the diagonals if for all $\omega \in \Omega=\Omega_\XX$ we have $f(x_1,\ldots,x_k; \, \omega)=0$ 
whenever $(x_1,\ldots,x_k)\in \XX^k_{diag}$, i.e. whenever $x_i=x_j$ for some $1\le i< j\le k$. 
This condition is restrictive only if $k\ge 2$.
We propose the following mixed Mecke-Palm formula.
\begin{lemma}\label{lem:multiPalm}
If $f\ge 0$ 
vanishes on the diagonals and $\epsilon_1, \ldots, \epsilon_k \in \{ 0,1 \}$, then
\begin{align}
\label{multiPalm}
\E\!\! \int\limits_{\XX^k}\!\! & \ f(x_1,\ldots, x_k; 
\omega) \,  \w_{\epsilon_1}(dx_1)\cdots \w_{\epsilon_k}(dx_k) 
\! \\
& =\!\!\int\limits_{\XX^k}\!\! \  \E \, f\Big(x_1,\ldots, x_k; 
\omega+ \sum_{i=1}^k \epsilon_i \delta_{x_i}\Big) \,\sigma(dx_1)\cdots\sigma( dx_k).
\nonumber
\end{align}
\end{lemma}
 
\begin{proof}
Case $k=0$ is trivial: $\E f(\w)=\E f(\w)$.
Case $k=1$ is precisely \eqref{Palm2}. 
For $k > 1$ we define
$$
g(x;\omega; \epsilon_1, \ldots, \epsilon_{k-1}) = \int_{\XX^{k-1}} f(x_1,\ldots, x_{k-1}, x; \omega) \,  \w_{\epsilon_1}(dx_1)\cdots \w_{\epsilon_{k-1}}(dx_{k-1}).
$$
Since 
$f(x_1,\ldots, x_{k
-1},x;\w)$ vanishes on $\XX^k_{diag}$, we get
\begin{align}\label{v}
&g(x; \omega+ \delta_x; \epsilon_1, \ldots, \epsilon_{k-1}) \nonumber \\ 
&= \int_{\XX^{k-1}} f(x_1,\ldots, x_{k-1}, x; \omega+ \delta_x) \,  (\w_{\epsilon_1}+ \epsilon_1 \delta_x)(dx_1)\cdots (\w_{\epsilon_{k-1}}+ \epsilon_{k-1}\delta_x)(dx_{k-1}) \nonumber \\
& = \int_{\XX^{k-1}} f(x_1,\ldots, x_{k-1}, x; \omega+ \delta_x) \,  \w_{\epsilon_1}(dx_1)\cdots \w_{\epsilon_{k-1}}(dx_{k-1}).
\end{align}
By \eqref{Palm2}, \eqref{v} and induction we obtain
\begin{align*}
 &  \E\int\limits_{\XX^k}\!\! \ f(x_1,\ldots, x_k; \omega) \,  \w_{\epsilon_1}(dx_1)\cdots \w_{\epsilon_k}(dx_k)  \\
 & = \E \int_{\XX} g(x_{k}; \omega; \epsilon_1, \ldots, \epsilon_{k-1}) ~ \w_{\epsilon_k}(dx_k) =  \int_{\XX} \E g(x_{k}; \omega + \epsilon_k  \delta_{x_k}; \epsilon_1, \ldots, \epsilon_{k-1}) ~\sigma(dx_{k})  \\
 & =  \int\limits_{\XX} \E \int_{\XX^{k-1}} f(x_1,\ldots, x_{k-1}, x_k; \omega + \epsilon_k  \delta_{x_k}) \,  \w_{\epsilon_1}(dx_1)\cdots \w_{\epsilon_{k-1}}(dx_{k-1})
\sigma( dx_{k}) \\ 
& \!=\!\!\int\limits_{\XX^k}\!\! \  \E \, f\Big(x_1,\ldots, x_k; \omega + \sum_{i=1}^k \epsilon_i \delta_{x_i}\Big) \,\sigma(dx_1)\cdots\sigma( dx_k),
\end{align*} 
which proves \eqref{multiPalm}.
\end{proof}

\begin{remark}
{\rm 
Lemma~\ref{lem:multiPalm} extends to signed processes $f$ satisfying
$$\int\limits_{\XX^k}\!\! \  \E \, \Big|f\Big(x_1,\ldots, x_k; 
\omega+ \sum_{i=1}^k \epsilon_i \delta_{x_i}\Big)\Big| \,\sigma(dx_1)\cdots\sigma( dx_k)<\infty,$$ because of the decomposition $f = f_+ - f_-$, where $f_{+} = \max (f,0)$ and $f_{-} = \max (-f,0)$. In what follows we leave such extensions to the reader.
}
\end{remark}

\begin{remark}\label{rem:ana}
   {\rm 
The assumption in Lemma \ref{lem:multiPalm}  that $f$ should vanish on the diagonals is essential.  Indeed,  take $k=2$ and (deterministic) $f(x_1,x_2; \w) = \mathbf{1}_{\{ x_1=x_2\}}$ for $(x_1,x_2) \in \XX^2$. 
Considering the atoms of $\w$ we have 
\begin{align*}
\int\limits_{\XX^2}  f(x_1,x_2; \w) ~\omega(dx_1)\omega(dx_2) = 
\sum_{x_1\in \w}\sum_{x_2\in \w} \1_{x_1=x_2}=\w(\XX),  
\end{align*}
hence
$$\E\int\limits_{\XX^2}  f(x_1,x_2; \w) ~\omega(dx_1)\omega(dx_2) = 
\sigma(\XX).$$  
On the other hand  $\sigma$ is non-atomic, therefore
$$\int\limits_{\XX^2} \E f(x_1,x_2; \w) ~\sigma(dx_1)\sigma(dx_2) =\int\limits_{\XX^2} f(x_1,x_2; \w)\sigma(dx_1)\sigma(dx_2)=0.$$
   }
\end{remark}
Motivated by the above example we shall give a version of the multiple Mecke-Palm formula for processes which do not necessarily vanish on the diagonals. This calls for a notation that can handle {\it partitions}:
For integers $k,n\geq 1$ we consider a family of pairwise disjoint nonempty sets (blocks) of integers $P = \{ P_1, \ldots, P_k \}$,  such that 
$\bigcup_{i=1}^k P_i = \{1, \ldots, n \}$.
Thus, $P$ is a partition of $\{1, \ldots, n \}$.
We denote by $\Pn$ the set of all such partitions. We will use partitions to describe effects of interlaced Poisson integrations on the diagonals of $\XX^n$, in a manner which resembles the approach to multiple It\^o integrals and compensated Poisson integrals in \cite{MR2791919}.
For $P\in \Pn$ we let
$$
\XX_{P}^n = \bigg\{(x_1, \ldots, x_n) \in \XX^n: x_i = x_j~ 
\mbox{
iff
}~i,j\in P_s~\mbox{for some}~ s\in \{1, \ldots, k\}   \bigg\}.
$$
For $P = \{ P_1, \ldots, P_k \}\in \Pn$ and $y\in \XX_{\neq}^k$, we define $y^{[P]} = (y_1^{[P]}, \ldots, y_n^{[P]})$ by letting $y_i^{[P]} = y_j$ if $i\in P_j$.
We have, as in Remark~\ref{rem:ana},
\begin{align}\label{eq:ond}
&\int_{\XX^n} f(x_1,\ldots,x_n;\w)\w(dx_1)\cdots\w(dx_n)\nonumber\\
&=
\sum_{P=\{P_1,\ldots,P_k\}\in \Pn}\int_{\XX_{P}^n} f(x;\w)\w(dx_1)\cdots\w(dx_n)\nonumber\\
&=
\sum_{P=\{P_1,\ldots,P_k\}\in \Pn}\int_{\XX_{\neq}^k} f(y^{[P]};\w)\w(dy_1)\cdots \w(dy_k). 
\end{align}
As in Remark~\ref{rem:ana} we also note that for $n>1$ and all $\w$,
\begin{align}\label{eq:cpgp}
\int_{\XX^n}\1_{x_1=x_2=\ldots=x_n}\sigma(dx_1)\w(dx_2)\cdots\w(dx_n)=0,
\end{align}
because the first marginal of the product measure is non-atomic.  Therefore in view of generalizing \eqref{eq:ond} to interlaced integrations against $\w_1$ and $\w_0$, we propose the following notation. 
For $\e_1,\ldots,\e_n\in \{0,1\}$ we let $\e = (\e_1,\ldots,\e_n)$ and consider the family $\Pn^{\e}$ of all partitions $P=\{P_1,\ldots,P_k\}$ of $\{1,\ldots,n\}$ such that for every block $P_i \in P$ with $|P_i|>1$ we have $\e_j = 1$ for all $j\in P_i$. 
For $P\in \Pn^{\e}$ we let $\e^{[P]}=(\e_{1}^{[P]}, \ldots, \e_k^{[P]})$, where $\e_1^{[P]}= \e_{i_1}, \ldots, \e_k^{[P]}= \e_{i_k}$ and $i_1 \in P_1, \ldots,i_k \in P_k$. For $y = (y_1,\ldots,y_k) \in \XX^k$ we then let $y_{\e^{[P]}} = \{y_i:~ \e_i^{[P]} = 1\}$. In the following extension of Lemma \ref{lem:multiPalm} we write $x$ for  $(x_1,\ldots,x_n) \in \XX^n$ and $\sigma^k(dy)=\sigma(dy_1)\cdots\sigma(dy_k)$. The identity \eqref{multiPalmGen} below gives an algorithm to calculate expectations of Poisson integrals mixed with integrations against the control measure.
\begin{theorem}\label{lem:multiPalmDiag}
Let $\E$ be the expectation making configurations $\omega$ on $\XX$ a Poisson random measure with control measure $\sigma$.
For every $n$-process $f\ge 0$ 
and $\epsilon_1, \ldots, \epsilon_n \in \{ 0,1 \}$ we have
\begin{align}
\label{multiPalmGen}
&\E \int\limits_{\XX^n}  \ f(x; 
\omega) \,  \w_{\epsilon_1}(dx_1)\cdots \w_{\epsilon_n}(dx_n) 
\\
& =\sum_{P=\{P_1,\ldots,P_k\}\in\Pn^{\e}}~\int\limits_{\XX_{\neq}^k}\  \E \, f\left(y^{[P]}; 
\omega\cup y_{\e^{[P]}}\right) \,
\sigma^k(dy
).
\nonumber
\end{align}
\end{theorem}
\begin{proof}
By similar reasons as in \eqref{eq:ond}, and by Lemma~\ref{lem:multiPalm}, 
\begin{align}
&\E \int\limits_{\XX^n}  \ f(x; 
\omega) \,  \w_{\epsilon_1}(dx_1)\cdots \w_{\epsilon_n}(dx_n)  \nonumber
\\
& = \sum_{P=\{P_1,\ldots,P_k\}\in\Pn}\E\int\limits_{\XX_{P}^{n}}  f\left(x; 
\omega\right) \,\w_{\epsilon_1}(dx_1)\cdots \w_{\epsilon_n}(dx_n)\nonumber\\
& = \sum_{P=\{P_1,\ldots,P_k\}\in\Pn^{\e}} \E \int\limits_{\XX_{P}^n} f\left(x; 
\omega\right) \,\w_{\epsilon_1}(dx_1)\cdots \w_{\epsilon_n}(dx_n)\label{eq:nav}\\
&=\sum_{P=\{P_1,\ldots,P_k\}\in\Pn^{\e}}\E  \int\limits_{\XX_{\neq}^k} f\left(y^{[P]}; 
\omega\right) \,\w_{\epsilon^{[P]}}(dy) \nonumber\\
&=\sum_{P=\{P_1,\ldots,P_k\}\in\Pn^{\e}}\; \int\limits_{\XX_{\neq}^k} \E  f\left(y^{[P]};  
\omega \cup y_{\e^{[P]}}\right) \,\sigma^k(dy).
\nonumber
\end{align}
In \eqref{eq:nav} 
we use \eqref{eq:cpgp}
to eliminate $P\notin \mathcal{P}_n^\epsilon$.

\end{proof}

\section{Moments}\label{sec:a}
In this section we give applications of the mixed Mecke-Palm formula to expectations of products of stochastic integrals with respect to the Poisson random measure $\omega$ with control measure $\sigma$ on $\XX$ and probability $\PP$ and expectation $\E$. 

\subsection{General moment formulas}\label{sec:gmf}
Theorem~\ref{main:Theorem} below generalizes moment formulas of \cite{MR3021456}, \cite{MR3210036}. As we see in the proof, the result is equivalent to the mixed Mecke-Palm formula \eqref{multiPalmGen}
and is obtained after a simple {\it linearization} procedure.
Let $S$ be a finite set and $\XX^S=\{x: S\to \XX \}$. For $x\in \XX^S$ and $s\in S$ we write $x_s=x(s)$. We consider $\mathcal{P}(S)$, the class of all the partitions $P=\{P_1,\ldots,P_k\}$ of $S$.
Here (blocks) $P_1,\ldots, P_k$ are disjoint, and $\bigcup_{\alpha=1}^k P_{\alpha}=S$. 
Let $P \in \mathcal{P}(S)$ and consider the $P$-diagonal:
$$\XX_P^S=\{x\in \XX^S: x_s=x_t \mbox{ iff
there is $i\in \{1,\ldots,k\}$ such that } s,t\in P_{\alpha}\}.$$
For $\epsilon: S\to \{0,1\}$ we denote $\omega_\epsilon(dx)=\otimes_{s\in S}\omega_{\epsilon_s}(dx_{s})$. We note that $\omega_\epsilon$ vanishes on $X_P^S$ if
there is block $P_{\alpha}\in P$ with cardinality $|P_{\alpha}|>1$ 
and such that $\epsilon=0$ at some point of $P_\alpha$. This is so because the product measure has a non-atomic marginal.
The set of the remaining partitions will be denoted $\mathcal{P}^\epsilon(S)$. In particular, if $P\in \mathcal{P}^\epsilon(S)$ then $\epsilon$ is constant on every block of $P$, and we may define $\epsilon_{\alpha}^P:=\epsilon_s$ if $s\in P_{\alpha}$, $\alpha=1,\ldots,k$.
We denote $\e^P = (\e_1^P, \ldots, \e_k^P)$.
For $y=(y_1,\ldots,y_k)\in \XX^k$ we let $y^P_s=y_{\alpha}$ if $s\in P_{\alpha}$. Thus, $\epsilon^P\in \{0,1\}^k$ and $y^P\in \XX^S$.
For measurable $f: \XX^S\to \Rp$ we have
$$
\int_{\XX^P} f(x)\omega_\epsilon(dx)=\int_{\XX^k}f(y^P)\omega_{\epsilon_1}(dy_1)\cdots \omega_{\epsilon_k}(dy_k),
$$
which follows because $\omega$ is a sum of Dirac measures supported at different points.

Let $ l \in \N_+$ and $r_1, n_1, \ldots, r_l, n_l \ge 1$.
We define $$S=\{(\alpha, \beta, \gamma): 1\le \alpha\le l, 1\le \beta\le r_{\alpha}, 1\le \gamma\le n_\alpha\}.$$ 
If $1\le \alpha \le l$ and $1 \le \gamma \le n_{\alpha}$, then we let 
$$S_{\alpha,\gamma} = \{(\alpha, \beta, \gamma) \in S: 1\le \beta \le r_{\alpha}\}.$$
For $z\in \XX^S$ we write, as usual, $z_{S_{\alpha,\gamma}}$ for the restriction of $z$ to $S_{\alpha,\gamma}$. If $P=\{P_1,\ldots,P_k\}\in \mathcal{P}(S)$ and $y\in \XX^k$, then $y^P_{S_{\alpha,\gamma}}$ denotes $(y^P)_{S_{\alpha,\gamma}}$. In particular, $y^P_{S_{\alpha,\gamma}}\in \XX^{S_{\alpha,\gamma}}$.
 \begin{theorem}\label{main:Theorem}
 Let $\E$ be the expectation making $\omega$ a Poisson random measure on $\XX$ with control measure $\sigma$.
Let  $f_0,f_1, \ldots, f_l \ge 0$ be $0, r_1, \ldots, r_l$-processes, respectively. 
Let $\e_{(1)} \in \{0,1\}^{r_1}, \ldots, \e_{(l)} \in \{0,1\}^{r_l}$. For  $s=(\alpha, \beta, \gamma)\in S$
we define $\epsilon_s=\epsilon_{(\alpha)}(\beta)$.
Then,
  \begin{align}\label{moment_integr}
 & \E  \Bigg[ f_0(\w) \left( \int_{\XX^{
r_1}} f_1(y; \w)  \w_{\e_{(1)}}(dy) \right)^{n_1} \ldots
   \left(  \int_{\XX^{
r_l}}  f_l(y ; \w)  \w_{\e_{(l)}}(dy) \right)^{n_l} \Bigg] \\ 
  &=  \!\!\!\!\!\!\!\!\sum\limits_{P = \{P_1, \ldots, P_k \} \in \mathcal{P}^\epsilon(S) } \!\! \E \left[  \int\limits_{\XX^k_{\neq}} f_0(\w\cup \bigcup\limits_{\e_s=1} \{y_s^P\})
\prod_{\alpha=1}^l \prod_{\gamma=1}^{n_{\alpha}} f_{\alpha}(y_{S_{\alpha,\gamma}}^P; \w\cup \bigcup\limits_{\e_s=1} \{y_s^P\}) 
\sigma^k(dy) \right].\nonumber
  \end{align}  
\end{theorem}

\begin{proof}The first transformation in the calculation below we call linearization, and the last one follows from Theorem~\ref{lem:multiPalmDiag}:
 \begin{align} 
 & \E  \Bigg[ f_0(\w)\left( \int\limits_{\XX^{r_1}} f_1(y; \w)  \w_{\e_{(1)}}(dy) \right)^{n_1} \ldots
   \left(  \int\limits_{\XX^{r_l}}  f_l(y ; \w)  \w_{\e_{(l)}}(dy) \right)^{n_l} \Bigg] \nonumber\\
 & = \E \bigg[ f_0(\w) \int_{\XX^{r_1}} \ldots \int_{\XX^{r_1}} \prod_{\gamma=1}^{n_1}  f_1(y_{S_{1,\gamma}} ; \w) \w_{\e_{(1)}}(dy_{S_{1,1}})\ldots \w_{\e_{(1)}}(dy_{S_{1,n_1}}) \times \nonumber\\
 & \;\;\;\;\; \cdots  \times   \int_{\XX^{r_l}} \ldots \int_{\XX^{r_l}} \prod_{\gamma=1}^{n_l}  f_l(y_{S_{l,\gamma}} ; \w) \w_{\e_{(l)}}(dy_{S_{l,1}})\ldots \w_{\e_{(l)}}(dy_{S_{l,n_l}})   \bigg] \nonumber\\ 
 & = \E \bigg[ \int_{\XX^S} f_0(\w) \prod_{\alpha = 1}^l \prod_{\gamma=1}^{n_{\alpha}} f_{\alpha} (y_{S_{\alpha, \gamma}}; \w) \w_{\e_{(1)}} (dy_{S_{1,1}}) \ldots \w_{\e_{(1)}} (dy_{S_{1,n_1}}) \ldots\nonumber\\
 & \qquad \qquad \qquad \qquad\qquad \qquad \qquad\;  \ldots  \w_{\e_{(l)}} (dy_{S_{l,1}}) \ldots \w_{\e_{(l)}} (dy_{S_{l,n_l}}) \bigg]\nonumber\\
  &= \E \left[ \sum\limits_{P = \{P_1, \ldots, P_k \} \in \mathcal{P}(S) \;} \int\limits_{\XX^k_{\neq}} 
f_0(\w) \prod_{\alpha=1}^l \prod_{\gamma=1}^{n_{\alpha}} f_{\alpha}(y_{S_{\alpha,\gamma}}^P; \w) \w_{\e^P}(dy)
 \right]\nonumber\\
  &= \E 
\!\!\!\!\!\!\!  \sum\limits_{P = \{P_1, \ldots, P_k \} \in \mathcal{P}^{\epsilon}(S)\; } \int\limits_{\XX^k_{\neq}} f_0(\w\cup \bigcup\limits_{\e_s=1} \{y_s^P\})
\prod_{\alpha=1}^l \prod_{\gamma=1}^{n_{\alpha}} f_{\alpha}(y_{S_{\alpha,\gamma}}^P; \w\cup \bigcup\limits_{\e_s=1} \{y_s^P\}) \sigma^k(dy).
\nonumber
  \end{align} 
\end{proof}
In concrete computations one may either use Theorem~\ref{main:Theorem}, along with its somewhat heavy notation, or just follow its proof, i.e. use linearization and the mixed Mecke-Palm formula. For instance in Lemma~\ref{lem:22} below it is simpler to use the latter approach.

\subsection{Moment formulas for stochastic integrals of $1$-processes}\label{sec:mf1}
We first specialize to $1$-processes.
Let $k, l,n_1, \ldots, n_l \in \N=\{1,2,\ldots\}$ and $n=n_1 + \ldots + n_l$. For $j = 1, \ldots, l$ and $P = \{ P_1, \ldots, P_k \}\in \Pn$ we denote
$$
P_{i,j} = \{d\in P_i: \sum_{0<m<j} n_m < d \leq  \sum_{0<m \leq j} n_m \}.
$$
Let $|P_{i,j}|$ be the number of elements of $P_{i,j}$. 
\begin{corollary}\label{1proc:result}
For a random variable $f_0 \ge 0$ and $1$-processes $f_1, \ldots, f_l\ge 0$,
\begin{eqnarray}\label{momentForm}
&&\E\bigg[ f_0(\w) \bigg( \int_{\XX} f_1(x;\w) \w(dx) \bigg)^{n_1} \cdots \bigg( \int_{\XX} f_l(x;\w) \w(dx) \bigg)^{n_l} \bigg] \\ 
&=& \sum_{P \in \Pn} \E\int_{\XX^k} f_0(\w+\sum_{i=1}^k \delta_{y_i}) f_1^{|P_{1,1}|}(y_1;\w+\sum_{i=1}^k \delta_{y_i}) \cdots f_l^{|P_{1,l}|}(y_1;\w+\sum_{i=1}^k \delta_{y_i}) \times\nonumber\\
&&\times f_1^{|P_{k,1}|}(y_k;\w+\sum_{i=1}^k \delta_{y_i}) \cdots f_l^{|P_{k,l}|}(y_k;\w+\sum_{i=1}^k \delta_{y_i}) \sigma(dy_1) \ldots \sigma(dy_k).\nonumber
\end{eqnarray}
\end{corollary}
\begin{proof}
The result follows from Theorem \ref{main:Theorem}.
\end{proof}
For $l=1$ we recover  \cite[(1.2)]{MR3021456}:
\begin{align*}
&\E \bigg[ v(\w) \bigg(\int_{\XX} u(x;\w) \w(dx) \bigg)^n \bigg] \\
=&
\sum_{P=\{P_1,\ldots, P_k\} \in \Pn} \!\!\E\bigg[ \int_{\XX^k}  v(\w\cup y) u(y_1;\w\cup y)^{|P_1|} \ldots u(y_k;\w\cup y)^{|P_k|} 
\sigma(dy_1) \ldots \sigma(dy_k)\bigg],
\end{align*}
where $y=\{y_1,\ldots,y_k\}$ and $u\ge 0$ is a $1$-process.
With arbitrary $l$ we obtain an alternative proof of \cite[Theorem~3.1]{MR3210036} for random Poisson measures.
In passing we also refer the reader to recent papers \cite{MR3215537} and \cite{2013arXiv1310.3531B}.

\subsection{The second moment of stochastic integrals of $2$-processes}\label{sec:mf2}
Moments of arbitrary $k$-processes require formulas of increasing complexity,
but they are entirely explicit. Here is a telling example.
\begin{lemma}\label{lem:22}
If $f\ge 0$ is a $2$-process, then
\begin{align}\label{eq:22}
&\E \left(\int_{\XX^2}f(x_1,x_2;\w)\w(dx_1)\w(dx_2)\right)^2= 
\!\! \int_{\XX}\E f^2(x,x;\w \cup \{x\}) \sigma(dx)\\
&+2\int_{\XX_{\neq}^2} \E f(x,x;\w\cup \{x,y\})f(x,y;\w\cup \{x,y\}) \sigma(dx)\sigma(dy)\nonumber\\
&+2\int_{\XX_{\neq}^2} \E f(x,x;\w\cup \{x,y\})f(y,x;\w\cup \{x,y\}) \sigma(dx)\sigma(dy)\nonumber\\
&+\int_{\XX_{\neq}^2} \E f(x,x;\w\cup \{x,y\})f(y,y;\w\cup \{x,y\})\sigma(dx)\sigma(dy)\nonumber\\
&+\int_{\XX_{\neq}^2} \E f^2(x,y;\w\cup \{x,y\})\sigma(dx)\sigma(dy)\nonumber\\
&+\int_{\XX_{\neq}^2} \E f(x,y;\w\cup \{x,y\})f(y,x;\w\cup \{x,y\})\sigma(dx)\sigma(dy)\nonumber\\
&+2\int_{\XX_{\neq}^3} \E f(x,x;\w\cup \{x,y,z\})f(y,z;\w\cup \{x,y,z\}) \sigma(dx)\sigma(dy)\sigma(dz)\nonumber\\
&+2\int_{\XX_{\neq}^3} \E f(x,y;\w\cup \{x,y,z\})f(z,x;\w\cup \{x,y,z\}) \sigma(dx)\sigma(dy)\sigma(dz)\nonumber\\
&+\int_{\XX_{\neq}^3} \E f(x,y;\w\cup \{x,y,z\})f(x,z;\w\cup \{x,y,z\}) \sigma(dx)\sigma(dy)\sigma(dz)\nonumber\\
&+\int_{\XX_{\neq}^3} \E f(y,x;\w\cup \{x,y,z\})f(z,x;\w\cup \{x,y,z\}) \sigma(dx)\sigma(dy)\sigma(dz)\nonumber\\
&+\int_{\XX_{\neq}^4} \E f(x,y;\w\cup \{x,y,z,t\})f(z,t;\w\cup \{x,y,z,t\}) \sigma(dx)\sigma(dy)\sigma(dz)\sigma(dt)\nonumber.
\end{align}
\end{lemma}
\begin{proof}
By linearization,
\begin{align*}
&\left(\int_{\XX^2} f(x_1,x_2; \w) \w(dx_1) \w(dx_2) \right)^2 \\
=& \int_{\XX^4} g(x,y,z,t)\w(dx)\w(dy)\w(dz)\w(dt),
\end{align*}
where
$g(x,y,z,t;\w) = f(x,y;\w)f(z,t;\w)$. We will use Theorem \ref{lem:multiPalmDiag}.
The partitions involved have $k=1$, $2$, $3$ or $4$ blocks, because the number $4$ can be represented as the following sums:
$4$, $3+1$, $2+2$, $2+1+1$, $1+1+1+1$.
In particular, the partition of $\{1,2,3,4\}$ with only one block ($k=1$), namely $\{\{1,2,3,4\}\}$, contributes
\begin{align*}
& \E\int_{\XX} g(x,x,x,x;\w\cup \{x\}) \sigma(dx) =  
\int_{\XX}\E f^2(x,x;\w \cup \{x\}) \sigma(dx)
\end{align*}
to \eqref{eq:22}.
Then, partitions with $k=2$ blocks are of type $3+1$ and $2+2$. In the first case there are $4$ different partitions as there are $4$ different choices of the singleton. For instance, $P=\{\{1,2,3\},\{4\}\}$ contributes 
\begin{align*}
&\int_{\XX_{\neq}^2} \E g(x,x,x,y;\w\cup \{x,y\}) \sigma(dx)\sigma(dy)\\
&= \int_{\XX_{\neq}^2} \E f(x,x;\w\cup \{x,y\})f(x,y;\w\cup \{x,y\}) \sigma(dx)\sigma(dy) \nonumber
\end{align*}
to \eqref{eq:22}.
The contribution to \eqref{eq:22} from all the partitions of type $3+1$ are the $2$nd and the $3$rd terms on the right-hand side of (\ref{eq:22}). In the case $2+2$, $P=\{\{ 1,2\},\{3, 4\}\}$, 
contributes 
\begin{align*}
\int_{\XX_{\neq}^2} &\E f(x,x;\w\cup \{x,y\})f(y,y;\w\cup \{x,y\})\sigma(dx)\sigma(dy),
\end{align*}
to \eqref{eq:22}, and the contributions from all the partitions of type $2+2$ are precisely the $4$th through $6$th terms on the right-hand side of (\ref{eq:22}). 

For $k=3$ we have partitions of type $2+1+1$, e.g. $P= \{\{1,2\}, \{3\},\{4\}\}$, which contributes
\begin{align*}
\int_{\XX_{\neq}^3} \E f(x,x;\w\cup \{x,y,z\})f(y,z;\w\cup \{x,y,z\}) \sigma(dx)\sigma(dy)\sigma(dz),
\end{align*}
to \eqref{eq:22},
and all partitions of type $2+1+1$ result in the $7$th through $10$th terms on the right-hand side of (\ref{eq:22}). 
Finally, the partition into $k=4$ singletons yields
  \begin{align*}
\int_{\XX_{\neq}^4} \E f(x,y;\w\cup \{x,y,z,t\})f(z,t;\w\cup \{x,y,z,t\}) \sigma(dx)\sigma(dy)\sigma(dz)\sigma(dt).
\end{align*}
This finishes the verification of (\ref{eq:22}).
\end{proof}

We now investigate the second moment of mixed double stochastic integrals, the ones with respect to the random measures $\w \otimes \sigma$ and $\sigma \otimes \w$.

\begin{lemma}\label{Lem:mixed}
If $f \ge 0$ is a $2$-process, then
\begin{align}
&\E\left( \int_{\XX^2} f(x_1, x_2; \w) \w(dx_2) \sigma(dx_1) \right)^2 \nonumber\\  
&\qquad=\E\left( \int_{\XX^2} f(x_1, x_2; \w) \sigma(dx_1) \w(dx_2) \right)^2 \label{mixed:2proc2}\\
&\qquad = \E \int_{\XX_{\neq}^3} f(x,y; \w \cup \{ y \}) f(z, y; \w \cup \{ y \}) \sigma(dx)\sigma(dy)\sigma(dz)  \label{mixed:2proc1}\\
&\qquad + \E \int_{\XX_{\neq}^4} f(x,y; \w \cup \{ y, t \}) f(z, t; \w \cup \{ y,t \}) \sigma(dx)\sigma(dy)\sigma(dz)\sigma(dt). \nonumber
\end{align}
\end{lemma}
\begin{proof}
The equation \eqref{mixed:2proc2} follows from Fubini-Tonelli. Then the expectation in \eqref{mixed:2proc2} is written as
\begin{equation*}
\E\int_{\XX^4}f(x, y; \w) f(z, t; \w) \sigma(dx) \w(dy) \sigma(dz) \w(dt),
\end{equation*}
and by Theorem \ref{lem:multiPalmDiag} we get the equality \eqref{mixed:2proc1}, as in the proof of Lemma~\ref{lem:22}.
\end{proof}

\section{L\'evy systems}\label{s:Ls}
An important motivation for this work is due to the so-called L\'evy systems for L\'evy processes.
These are identities between expectations of sums taken with respect to the jumps of a L\'evy process and expectations of integrals taken with respect to the corresponding intensity measure. 
There exist a considerable variety of (multiple) L\'evy systems, as we discuss below.

\subsection{General result}\label{sec:gLs}

We consider (time) $\Rp=(0,\infty)$, (space) $\Rd$ and (space-time) $\Rp\times \Rd$. 
 
Let $\nu$ be a non-zero L\'evy measure on $\Rd$, thus $\nu (\{0 \}) = 0$ and
\begin{displaymath}
\int_{\R^d} \min \{1, z^2\} \nu(dz) < \infty.
\end{displaymath}
Let $X= \{X_t\}_{t\ge 0}$ be a L\'evy process in $\Rd$
with L\'evy triplet $(\nu,A,b)$, where $A$ is a symmetric, nonnegative-definite $d \times d$ matrix and $b \in \R^d$
\cite{MR1739520}. Let $\PP$ and $\E$ be the distribution and the expectation of the process and consider
$$p_t(A)=\Prawd(X_t\in A),$$ the convolution semigroup of $X$.
 Let $\Delta X_u = X_u - X_{u-}$ and
$$
\omega  = \sum_{u>0, \, \Delta X_u \ne 0} \delta_{(u, \Delta X_u)}.
$$
Then $\omega$ is a Poisson random measure with the intensity (control) measure $\sigma(du dz) = du\nu(dz)$ on $$\XX = \Rp \times \Rdz$$
\cite[Section~I.9, Section~II.3, Example~II.4.1]{MR1011252} related to
$X$ by 
the L\'evy-It\^o decomposition \cite[Chapter~4]{MR1739520}, \cite[Example~II.4.1]{MR1011252}.
We may and do identify $\omega$, $\PP$ and $\E$ with those from Section~\ref{sec:iMPf} given by $\sigma(dudz)=du\nu(dz)$.
The well-known (simple) L\'evy system is the following identity (more comments are given after the proof).
\begin{lemma}\label{lem2} 
If $F\,:\; \Rp\times \R^d\times \R^d\to\Rb$ is nonnegative,  then  
  \begin{equation}\label{eq1.3}
    \E \sum_{\substack{0< u < \infty \\ \dX_u \neq 0}} F(u,X_{u-}, X_{u})
= \E \int\limits_{0}^{\infty} \int\limits_{\R^d} F(u,X_{u},
X_{u}+z)\nu(dz)du\,.
  \end{equation}

\end{lemma}
\begin{proof}
First, let $X$ be a compound Poisson process, that is $\nu(\Rd)<\infty$, $X(t)=\sum_{i=1}^{N(t)}Z_i$, where $N(t)$ has Poisson distribution with expectation $t\nu(\Rd)$, and $Z_i$ are i.i.d. random variables with distribution $\nu/\nu(\Rd)$. Therefore
$$p_{t} = e^{-|\nu|t}e^{*t\nu} = e^{-|\nu|t} \sum\limits_{n=0}^{\infty} \frac{t^n \nu^{*n}}{n!}.$$
By Fubini-Tonelli theorem the right-hand side of (\ref{eq1.3}) equals
\begin{equation}\label{eq1.2}
\int\limits_0^{\infty} \int\limits_{\R^d} \int\limits_{\R^d} F(u,x,x+z) p_{u}(dx) \nu(dz)du.
\end{equation}
Let $S_i=\inf\{t>0: N(t)=i\}$, the arrival time of the $i$-th jump of $X$. 
Recall that $S_i$ has gamma distribution,
and clearly $X_{S_i}$ has distribution $\widetilde{\nu}^{*i}$. 
By Fubini-Tonelli the left-hand side of (\ref{eq1.3}) equals
\begin{eqnarray*}
&&\E \sum_{i=1}^{\infty}F(S_i,X_{S_i -}, X_{S_i})\\
 &=& \sum_{i=1}^{\infty} \int\limits_0^{\infty}\int\limits_{\R^d}\int\limits_{\R^d}  F(u,x,x+z)\frac{|\nu|^i u^{i-1}}{(i-1)!} 
e^{-|\nu|u}~ \widetilde{\nu}^{*(i-1)}(dx)\widetilde{\nu}(dz)du\\
 &=& \int\limits_0^{\infty} \int\limits_{\R^d}\int\limits_{\R^d}F(u,x,x+z)\bigg(e^{-|\nu|u}\sum_{i=1}^{\infty}\frac{u^{i-1}\nu^{*(i-1)}}{(i-1)!}(dx)\bigg)
\nu(dz)du\\
 &=& \int\limits_0^{\infty} \int\limits_{\R^d}\int\limits_{\R^d} F(u,x,x+z)p_{u}(dx) \nu(dz)du.
\end{eqnarray*}
This yields (\ref{eq1.3}) for compound Poisson process $X$.
Now let $X$ be a general L\'evy process. We shall prove that for every $\e > 0$,
\begin{equation}\label{eq1.4}
\E \sum\limits_{\substack{0<u < \infty \\ |\dX_u| \geq \e}}F(u, X_{u-}, X_{u}) = \E \int_{0}^{\infty} \int_{|z| \geq \e} F(u,X_{u},
X_{u}+z)\nu(dz)dv.
\end{equation}

\noindent To this end we use the following decomposition,
\begin{displaymath}
X_t = V_t +Z_t.
\end{displaymath}
The terms in the decomposition have the following properties. 
\noindent Process $V_t$ is a L\'evy process with the triplet $(A, \nu \big|_{|z| < \e},b)$, on a probability space $(\Omega^V, \F^V, \Prawd^V)$.
Here $\nu \big|_{|z| < \e}$ is the measure $\nu$ restricted to $\{z \in \R^d: |z|< \e \}$.
$Z_t$ is a compound Poisson process on an independent probability space $(\Omega^Z, \F^Z, \Prawd^Z)$, and has the L\'evy measure $\nu \big|_{|z| \geq \e}$. We denote by $\E^V, \E^Z$ and $\Prawd^V$, $\Prawd^Z$ the corresponding expectations and probabilities. We may assume that $\Omega = \Omega^V \times \Omega^Z$ and $\Prawd = \Prawd^V \otimes \Prawd^Z$, 
according to the fact that $V$ and $Z$ are independent. In what follows we consider
\begin{equation}\label{eq:dec}
\widetilde{F}(v,x,y) = F(v, V_{v-} + x, V_v + y).
\end{equation}
By Fubini-Tonelli theorem and by (\ref{eq1.3}) for the compound Poisson process $Z$, the left hand side of (\ref{eq1.4}) becomes
\begin{eqnarray*}
&& \E^V \E^Z \sum\limits_{ |\Delta(V_u+Z_u)| \geq \e} F(u, V_{u-} +Z_{u-}, V_{u} + Z_{u})\\ 
&=& \E^V \E^Z \sum\limits_{ |\Delta Z_u| \geq \e}\widetilde{F}(u, Z_{u-}, Z_{u})
= \E^V \E^Z \int\limits_{0}^{\infty} \int\limits_{|z| \geq \e} \widetilde{F}(u, Z_{u}, Z_{u}+z) \nu(dz)du\\
&=& \E \int\limits_{0}^{\infty} \int\limits_{|z| \geq \e} F(u, X_{u}, X_{u} +z) \nu(dz)du.
\end{eqnarray*}
We have proved (\ref{eq1.4}). Let $\e \downarrow 0$.
By monotone convergence theorem,
\begin{equation}\label{eq1.5}
\E \sum\limits_{ |\dY_u| \geq \e}F(u, X_{u-}, X_{u}) \rightarrow \E \sum\limits_{\dY_u \neq 0 }F(u, X_{u-} ,X_{u}),
\end{equation}
and 
\begin{equation}\label{eq1.6}
\E \int\limits_{0}^{\infty} \int\limits_{|z| \geq \e} F(u, X_{u},X_{u}+z)\nu(dz)du \rightarrow \E \int\limits_{0}^{\infty} \int\limits_{\R^d} F(u,X_{u},X_{u}+z)\nu(dz)du.
\end{equation}
\noindent By (\ref{eq1.6}), (\ref{eq1.5}) and (\ref{eq1.4}) we obtain (\ref{eq1.3}).
\end{proof}
Lemma~\ref{lem2} asserts that 
the expected sum over the jumps of the L\'evy process $X$ equals to the expectation of the integral with respect to the corresponding intensity measure. As we remarked, the result is well-known, see \cite{MR0343375}, \cite[p. 375]{MR745449}, \cite[VII.2(d)]{MR1138461}, but the above direct proof seems original and
will be referred to below in extensions which we call multiple mixed L\'evy systems.
Before presenting them we propose a reformulation of Lemma \ref{lem2}.
\begin{lemma}\label{lem2a} 
If $F\,:\; \Rp\times \R^d\times \R^d\to\Rb$ is nonnegative, then  
  \begin{equation*}
    \E \sum_{\substack{0< u < \infty \\ \dX_u \neq 0}} F(u,X_{u-}, \dX_{u})
= \E \int\limits_{0}^{\infty} \int\limits_{\R^d} F(u,X_{u},
z)\nu(dz)du\,.
  \end{equation*}
\end{lemma}
Here $\Rp=(0,\infty)$. The multiple mixed L\'evy systems can be described within the framework presented in the previous sections.
We consider the ``simplex''
$$
\XX_{<}^n=\{(u_1,z_1; \dots; u_n,z_n) \in \XX^n: 0<u_1<\cdots<u_n \}.
$$
The following defines the (complete set of the multiple) mixed L\'evy systems.
\begin{theorem}\label{E} Let $X$ be a L\'evy process in $\Rd$ with the L\'evy measure $\nu$, the expectation $\E$ and the Poisson random measure of jumps $\omega$.
Let $\epsilon_1, \ldots, \epsilon_n \in \{ 0,1\}$ and let $F: ( \Rp \times\Rd \times \Rd)^n \mapsto [0,\infty]$ be measurable. Then,
\begin{align} 
\label{Sk}
  &\E \int_{\XX^n_{<}} F( u_1,X_{u_1-}, z_1; \dots;  u_n,X_{u_n-}, z_n)\w_{\epsilon_1}(du_1dz_1) \ldots \w_{\epsilon_n}(du_ndz_n)\\
=&\int_{\XX^n_{<}} \E F( u_1,X_{u_1-},z_1; \dots; 
  u_j,X_{u_j-} + \sum_{i=1}^{j-1} \epsilon_i  z_i ,\ z_j; \dots; \nonumber\\
 & \qquad u_n,X_{u_n-} + \sum_{i=1}^{n-1}  \epsilon_i  z_i ,\ z_n) \;  du_1  \nu(dz_1)\cdots du_n \nu(dz_n)\nonumber\\
=&\int_{\XX^n_{<}} \int_{(\Rd)^n}   F( u_1,y_1,z_1; \dots; 
 u_n,\sum_{i=1}^{n}y_i + \sum_{i=1 }^{n-1} \epsilon_i  z_i , z_n) \nonumber\\
&\qquad  p_{u_1}(dy_1)\ldots p_{u_n-u_{n-1}}(dy_n)\,du_1 \nu(dz_1)\cdots du_n \nu(dz_n). 
\label{eq:aaa3}
\end{align}
\end{theorem}
\begin{proof}
We first prove this result for compound Poisson process $X$.
By the L\'evy-It\^o decomposition 
for $t\ge 0$ we have
$$X_{t-} = X_{t-}(\w) = 
\int\limits_{(0,t)}\int\limits_{\Rd} z \omega(dudz) -t\nu(\Rd),$$
and
$$X_t = X_t(\w) = 
\int\limits_{(0,t]}\int\limits_{\Rd} z \omega(dudz) -t\nu(\Rd).$$
We note that $X_{t-}$ is a $1$-process on $\XX$, and $$\1_{\XX^n_{<}}( u_1, z_1; \dots;  u_n, z_n)F( u_1,X_{u_1-}, z_1; \dots;  u_n,X_{u_n-}, z_n)$$ is an $n$-process, which vanishes on the diagonals.
 Using the notation from the proof of Lemma \ref{lem:multiPalm}, by Theorem~\ref{lem:multiPalmDiag} we see that the left-hand side of \eqref{Sk} equals 
\begin{align*}
\int_{\XX^n_{<}} \E F\big(&u_1,X_{u_1-}(\omega  + \sum_{i=1}^n \epsilon_i \delta_{(u_i,z_i)}  ) , z_1; \dots; u_n, X_{u_n-}(\omega  + \sum_{i=1}^n \epsilon_i \delta_{(u_i,z_i)}), z_n\big) \\
&du_1 \nu(dz_1)\cdots du_n \nu(dz_n).
\end{align*}
Since
$$
X_{u_j-}(\omega  + \sum_{i=1}^j \epsilon_i \delta_{(u_i,z_i)}) = X_{u_j-}(\w) + \sum_{i=1}^{j-1}\epsilon_i  z_i,
$$
\eqref{Sk} follows. Then we note that the distribution of $X_{u-}$ is the same as that of $X_u$, which is
$p_u$, and we use Fubini-Tonelli to get \eqref{eq:aaa3}. This resolves the case of compound Poisson processes.
The case of general L\'evy processes follows as in the proof of Lemma~\ref{lem2}.
\end{proof}

\begin{corollary}\label{lem4}
If $X$ is a L\'evy process and  $F$ is nonnegative, then
 \begin{eqnarray}\label{eq2.3}
&&\E{\sum_{\substack{0< u_1<\ldots<u_n \leq \infty \\ \dX_{u_1} \neq 0, \ldots, \dX_{u_n} \neq 0}} F(u_1,X_{u_1-}, X_{u_1};\ldots; u_n,X_{u_n-}, X_{u_n})}\\
&&= \E \int\limits_0^{\infty} \ldots \int\limits_{u_{n-1}}^{\infty} \int\limits_{\R^d} \ldots \int\limits_{\R^d} F(u_1,X_{u_1}, X_{u_1}+z_1;\ldots; \nonumber\\
&&u_n,X_{u_n}+z_1+\ldots+z_{n-1}, X_{u_n}+z_1+\ldots+z_n)\nu(dz_n)\ldots\nu(dz_1)du_n \ldots du_1. \nonumber
  \end{eqnarray}
\end{corollary}

\begin{corollary}\label{lem6}If $X$ is a L\'evy process and $F$ is nonnegative, then  
\begin{eqnarray}\label{eq3.5}
    &&\E{\sum_{\substack{0< s < \infty \\ \dY_{s} \neq 0}} \int_{s}^{\infty} \int_{\R^d}
F(s,X_{s-}, X_{s};s_1,X_{s_1}, X_{s_1}+z_1)\nu(dz_1)ds_1}\\
&=& \E \int_0^{\infty}  \int_{\R^d} \sum_{\substack{s<s_1 < \infty \\ \dX_{s_1} \neq 0}} 
F(s,X_{s}, X_{s}+z; s_1,X_{s_1-}+z, X_{s_1}+z)\nu(dz)ds\nonumber\\
&=& \E \int\limits_0^{\infty} \int\limits_{s}^{\infty} \int\limits_{\R^d}\int\limits_{\R^d}
F(s,X_{s}, X_{s}+z; s_1,X_{s_1}+z, X_{s_1}+z+z_1) \nu(dz_1)\nu(dz)ds_1 ds. \nonumber
  \end{eqnarray}
\end{corollary}
We note in passing that Corollary~\ref{lem4} and Corollary~\ref{lem6} can also be proved without using Mecke-Palm formula, in a way similar to the first part of the proof of Lemma~\ref{lem2}, see \cite{dissert}. The proofs are quite involved 
and the proof of the general mixed L\'evy systems is fraught with problems if similar approach is to be used.
On the contrary, Theorem~\ref{E} offers a clear insight into the structure of multidimensional mixed-type L\'evy systems. The structure is explained by accumulating $z_i$, the $i$-th variable of the integrations performed in
\eqref{eq:aaa3}, as a jump of the process $X$ at the moment $u_i$,  but only if 
$z_i$ is integrated against the Poisson random measure, rather than it's control measure.
By accumulation we mean that such jumps are indeed added to the trajectory of the process.
We encourage the reader to consider the statement of Corollary~\ref{lem6} from this perspective.
Notably, the complex machinery of stochastic analysis of general Markov processes, e.g., the notion of predictability  plays little role in the above treatment of L\'evy systems for L\'evy processes.
\begin{remark}\label{rem:gLs}
{\rm 
We note that Theorem~\ref{E} may be generalized to allow for $n$-processes more complicated than
$F( u_1,X_{u_1-}, z_1; \dots;  u_n,X_{u_n-}, z_n)$, with similar proofs based on the mixed Mecke-Palm formula.
Such extensions may involve modifications by predictable factors, cf. \cite[p. 375]{MR745449}, \cite[VII.2(d)]{MR1138461}, and integrating processes which are not adapted to the usual filtration associated with the L\'evy process. 
}
\end{remark}
To illustrate Remark~\ref{rem:gLs} we give the following classical extension, cf. \cite[VII.2(d)]{MR1138461}.
An additional discussion is given at the end of Section~\ref{sec:asLs}.
\begin{lemma}\label{lempr} 
If $F\ge 0$ and $g_t\ge 0$ is predictable, then  
  \begin{equation}
    \E \sum_{\substack{0< u < \infty \\ \dX_u \neq 0}} g_uF(u,X_{u-}, X_{u})
= \E \int\limits_{0}^{\infty} \int\limits_{\R^d} g_u F(u,X_{u},
X_{u}+z)\nu(dz)du\,.
  \end{equation}
\end{lemma}
\begin{proof}
$g_uF(u,X_{u-}, X_{u})=g_u(\w)F(u,X_{u-}(\w), X_{u-}(\w)+z)$ is a $1$-process on $\Rp\times \Rd$.  By predictability, $g_u(\w+\delta_{(u,z)})=g_u(\w)$ with almost surely. The result follows from the usual Mecke-Palm identity \eqref{mSi}.
\end{proof}

\subsection{Applications}\label{sec:asLs}
The purpose of this section is to show usability of our formulas.
A typical application of the L\'evy system is to the well-known Ikeda-Watanabe formula \cite{MR0142153}, given as \eqref{eq:IW} below. 
The formula concerns the situation of the  L\'evy process $X$ in $\Rd$ as it reaches the complement of
the open set $D\subset \Rd$.
We shall use the usual Markovian notation: for $x\in \Rd$ we write $\E^x$ and $\PP^x$ for the expectation and distribution of $x+X$, but we use the same symbol $X$ for the resulting process, cf. \cite[Chapter~8]{MR1739520}.
We write $p_t(x, A)=p_t(A-x)=\PP^x(X_t\in A)$, so that
$$
\E^x \int_0^\infty f(t,X_t)dt=\int_0^\infty \int_\Rd f(t,y)p_t(x,dy)
$$
for (Borel) functions $f\ge 0$ and $x\in \Rd$.
We consider the time of the first exit of $X$ from $D$,
$$\tau_D=\inf\{t>0: \, X_t\notin D\}.$$
The Dirichlet kernel $p_t^D(x,dy)$ is defined by
$$
\int_\Rd f(y)p_t^D(x,dy)=\E^x [f(X_t);\tau_D>t],
$$
and we have
$$
\E^x \int_0^{\tau_D} f(t,X_t)dt=\int_0^\infty \int_\Rd f(t,y)p_t^D(x,dy).
$$
We now consider function $F(u,y,w)=\1_{I}(u)\1_A(y) \1_B(w)$, where $I$ is a bounded interval, and $A\subset D$, $B\subset (\overline{D})^c$ are Borel sets in $\Rd$. 
We let
\begin{equation*}
M(t)= \sum\limits_{\substack{0<u \le t \\ |\dX_u| \neq 0}}F(u, X_{u-}, X_{u}) - \int_{0}^{t} \int_{\Rd} F(u,X_{u},
X_{u}+z)\nu(dz)dv.
\end{equation*}
We note that \
\begin{equation}\label{eq:ui}
\E \int_{0}^{t} \int_{\Rd} F(u,X_{u},X_{u}+z)\nu(dz)dv\le |I|\nu(\{|z|>{\rm dist}(A,B)\})<\infty,
\end{equation} 
so by Lemma~\ref{lem2}, $\E M(t)=0$.  
Let $0\le s\le t$. By considering the L\'evy process $u\mapsto X_{s+u}-X_s$, independent of $X_r$, $0 \le r \le s$, we calculate the conditional expectation
$$\E[\sum\limits_{\substack{s<u \le t \\ |\dX_u| \neq 0}}F(u, X_{u-}, X_{u}) - \int_{s}^{t} \int_{\Rd} F(u,X_{u},
X_{u}+z)\nu(dz)dv\big| X_r, 0\le r\le s]=0,$$ cf. \eqref{eq:dec}.
Then we see that $M$ is a uniformly integrable martingale. By stopping at $\tau_D$,   
we obtain, 
\begin{align}\label{eq:IW}
\PP^x[\tau_D\in I,\; X_{\tau_D-}\in A,\; X_{\tau_D}\in B]=
\int_I \int_{B-y} \int_A p_u^D(x,dy)\nu(dz)du.
\end{align}
This defines the joint distribution of $(\tau_D,X_{\tau_D-},X_{\tau_D})$ restricted to the event $\{X_{\tau_D-}\in D\}$ and calculated under $\PP^x$.

As another application we use the double mixed L\'evy system to prove the following classical result 
\cite[II (3.9)]{MR1011252}.

\begin{lemma}\label{martingaleLevy} Let $X$ be a L\'evy process in $\Rd$ with L\'evy measure $\nu$.
Let the function $F: \R \times \R^d \times \R^d \rightarrow \Rb$ satisfy
\begin{equation}\label{eq4.1}
\E \int\limits_0^{\infty} \int\limits_{\R^d} F^2(v,X_{v},X_{v}+z) \nu(dz) dv < \infty.
\end{equation}
For every $t \in [0,\infty)$ the following limit exists in $L^2$
$$M_{t} = \lim\limits_{\e \rightarrow 0} \bigg( \sum_{\substack{0< v \leq t \\ |\dX| \geq \e}} F(v,X_{v-}, X_{v})- \int_0^t \int_{|z| \geq \e} F(v,X_{v},X_{v}+z)\nu(dz)dv \bigg),$$
$t\mapsto M_{t}$ is a martingale with respect to $(\F_t)$, $\E M_t=0$ and $$\E M_t^2=\E\int_0^{t} \int_{\R^d} F^2(v,X_{v},X_{v}+z) \nu(dz) dv.$$ Furthermore, the square bracket of $M$ is
\begin{equation}\label{eq:wsb}
[M]_{t}= \sum\limits_{\substack{0< v \leq t\\ \dX_v\neq 0}}F^2(v,X_{v-},X_{v}),
\end{equation}
and the predictable quadratic variation of $M$ is
\begin{equation}\label{eq:wqv}
\langle M\rangle_t=\int_0^t \int_{|z| \geq \e} F(v,X_{v},X_{v}+z)^2\nu(dz)dv.
\end{equation}
\end{lemma}

\noindent Recall that $[M]$ is defined as the unique adapted right-continuous non-decreasing process with jumps $\Delta[M]_t=|\Delta M_t|^2$, and such that $t\mapsto |M|^2_t - [M]_t$ is a (continuous) martingale starting at $0$ (\cite[VII.42]{MR745449}). We verify the martingale property of $|M|^2_t - [M]_t$ by using Corollary \ref{lem4} and Corollary \ref{lem6}. Notice that $\E [M]_t = \E \langle M \rangle_t$ by the property of a single L\'evy system. More details and applications can be found in \cite{dissert}. In particular, the square bracket $[M]$ is used in \cite{pre06142552} to estimate the $L^p$ norms of Fourier multipliers defined in terms of L\'evy processes. We refer the reader to \cite[VII-VIII]{MR745449} and \cite[II]{MR1011252} for further details and reading.

As the third application 
we will calculate moments of the L\'evy integral.
Let $X_t=(\eta_t,\xi_t)$, where $t\ge 0$, be a  L\'evy process in $\R^2$.
To simplify the discussion we further assume that  $\eta$ and $\xi$ are (possibly dependent) subordinators with no drift  \cite{MR1739520,MR2598208}. Let $\nu$ be the L\'evy measure of $X$.
Of course, $\nu$ is concentrated on $\R_{++}^2:=(0,\infty)\times (0,\infty)$.
Let $\phi$ be the Laplace exponent of $\eta$:
$$\E\left[e^{-x\eta_t}\right]=e^{-t\phi(x)},\quad x\geq 0.$$
  The following  expression is called the L\'evy integral,
$$Z=\int_0^\infty e^{-\eta_t-}d\xi_t=\sum_{\Delta X_t\neq 0}e^{-\eta_{t-}}\Delta \xi_t.
$$
L\'evy integrals  represent stationary distributions of generalized Ornstein-Uhlenbeck process (see \cite{MR2165340} for details, applications and references). By Lemma~\ref{lem2},  
$$\E\left[\xi_1\right]=\E \sum\limits_{\substack{ 0<t\le 1\\\Delta X_t\neq 0}}\Delta \xi_t=\int_{\R_{++}^2} y\,d\nu(x,y).$$
We can use the multiple L\'evy systems to calculate the moments of $Z$. 
The first three moments of $Z$ take on the following form
\begin{align*}
\E [Z]&=\frac{\int y\,d\nu(x,y)}{\phi(1)},\\
\E [Z^2]&=\frac{2\int e^{-x}y\,d\nu(x,y)\int y\,d\nu(x,y)}{\phi(1)\phi(2)}+\frac{\int y^2\,d\nu(x,y)}{\phi(2)},\\
\E [Z^3]&=\frac{6\int y\,d\nu(x,y)\int e^{-x}y\,d\nu(x,y)\int e^{-2x} y\,d\nu(x,y)}{\phi(1)\phi(2)\phi(3)}+\frac{\int y^3\,d\nu(x,y)}{\phi(3)}\\
&\ \ \ +
\frac{3\int y^2\,d\nu(x,y)\int e^{-2x}y\,d\nu(x,y)}{\phi(2)\phi(3)}+\frac{3\int y\,d\nu(x,y)\int e^{-x}y^2\,d\nu(x,y)}{\phi(1)\phi(3)}.
\end{align*}
Indeed, by Lemma~\ref{lem2},  
\begin{align*}
\E [Z]&=\E\left[\sum_{\Delta X_t\neq 0}e^{-\eta_{t-}}\Delta \xi_t\right]=\E\left[\int_0^\infty \int
e^{-\eta_{t}}y\,d\nu(x,y)dt\right]
=\frac{\int y\,d\nu(x,y)}{\phi(1)}.
\end{align*}
For the higher moments we use linearization, as in Section~\ref{sec:a}, e.g., we obtain
\begin{align*}
\E [Z^2]&=\E\left[\left(\sum_{\Delta X_t\neq 0}e^{-\eta_{t-}}\Delta \xi_t\right)^2\right]
=\E\left[
\left(\sum_{\Delta X_s\neq 0}e^{-\eta_{s-}}\Delta \xi_s\right)
\left(\sum_{\Delta X_t\neq 0}e^{-\eta_{t-}}\Delta \xi_t\right)
\right]
\\
&=\E\left[2\sum_{\substack{s<t\\\Delta X_s,\Delta X_t\neq 0}}e^{-\eta_{s-}-\eta_{t-}}\Delta \xi_s\Delta \xi_t\right]+\E\left[\sum_{\Delta X_t\neq 0}\left(e^{-\eta_{t-}}\Delta \xi_t\right)^2\right]=2\rm{I}+\rm{II},
\end{align*}
where, by Corollary~\ref{lem4},
\begin{align*}
\rm{I}
&=\E\left[\int_0^\infty \int_s^\infty\int \int e^{-\eta_s}y_1 e^{-\eta_{t}-x_1}y_2 \,d\nu(x_1,y_1)\,d\nu(x_2,y_2)dt\,ds\right]\\
&=\int \int e^{-x_1}y_1y_2\,d\nu(x_1,y_1) \,d\nu(x_2,y_2)\,\E\left[\int_0^\infty \int_s^\infty \,e^{-(\eta_{t}-\eta_{s})-2\eta_{s}}dt\,ds\right]\\
&=\int e^{-x}y\,d\nu(x,y)\int y\,d\nu(x,y)\,\int_0^\infty \int_s^\infty \,\E\left[e^{-\eta_{t-s}}\right]\E\left[e^{-2\eta_{s}}\right]dt\,ds\\
&=\frac{\int e^{-x}y\,d\nu(x,y)\int y\,d\nu(x,y)}{\phi(1)\phi(2)},
\end{align*}
and
\begin{align*}
\rm{II}
&=\E\left[\int_0^\infty\int y^2 e^{-2\eta_{t}}\,d\nu(x,y)dt \right]=\frac{\int y^2\,d\nu(x,y)}{\phi(2)}.
\end{align*}
The third and the higher moments are obtained analogously. We note that  \cite[Theorem 3.1]{MR2858217}  gives the first and the second moments of $Z$, but not the higher moments, which are cumbersome to obtain by the methods of \cite{MR2858217} (private communication).  

By our methods one also compute moments of anticipating integrals such as
$$Y:=\int_0^\infty e^{-\eta_t}d\xi_t=\sum_{\Delta X_t\neq 0}e^{-\eta_{t-}-\Delta \eta_t}\Delta \xi_t.$$
Here, similar calculations as for $Z$ yield
\begin{align*}
&\E [Y]=\frac{\int e^{-x}y\,d\nu(x,y)}{\phi(1)},
\end{align*}
and higher moments of $Y$ can be obtained analogously.
Notice the difference between the formulas for
the expectations of $Z$ and $Y$.

\section{Appendix}\label{sec:MP}
The following Mecke-Palm identity holds for $1$-processes $f(x;\w)\ge 0$ \cite{MR2531026},
\begin{equation}\label{mSi}
\E  \int_\X f(x;\w)\ \w (dx)
= \int_{\X}\E \ f(x;\w \cup \{x\})\ \sigma( dx).
\end{equation}
For the reader's convenience we give a direct proof of 
\eqref{mSi} in the setting of Section~\ref{sec:iMPf}.
We first consider $\sigma(\X)<\infty$ and nonnegative process $f(x;\w)=f(x;\w\cap \X)$, i.e. depending only on $\X$. 
If $\w = \{y_1, \ldots, y_n\}$, a set with $n$ elements, then 
$$\int_\X f(x;\w) \w(dx)=\sum_{i=1}^n f_{(n)} (y_i; y_1, \ldots, y_n)
.$$
The above quantity 
is invariant upon permutations of 
$y_1, \ldots, y_n$, in fact it is the $n$-th coefficient of the 
random variable $\int_\X f(x;\w) \w(dx)$.
By \eqref{emSie}, the left-hand side of \eqref{mSi} equals
\begin{equation}\label{emSia}
e^{-\sigma(\X)}\sum_{n=1}^{\infty}\frac{1}{n!} \sum_{i=1}^n  \int_{\X^n}f_{(n)}(y_i;y_1, \ldots, y_n) \sigma(dy_1)\cdots \sigma(dy_n).
\end{equation}
If $\w = \{y_1, \ldots, y_n\}$, a set with $n$ elements, and $x\not \in \omega$, then 
\begin{align*}
f(x;\w \cup \{x\})&=f_{n+1} (x;x,y_1,\ldots, y_n)=\ldots=f_{n+1} (x;y_1,\ldots, y_n,x)\\
&=\frac1{n+1}\sum_{i=1}^{n+1}f_{n+1} (x;y_1,\ldots,y_{i-1},x,y_{i},\ldots, y_n).
\end{align*}
Since $\sigma$ is non-atomic, we have $\PP(x\in \w)=0$ for every $x\in \X$, cf. \eqref{emSie}. 
Therefore, by  \eqref{emSie},
the right-hand side of \eqref{mSi} equals
\begin{align*}
&\E \int_{\X} u(x;\w \cup \{x\})\sigma( dx)\\
&= e^{-\sigma(X)}\sum_{n=0}^{\infty}\frac{1}{n!}  
\int_\X \int_{\X^n}\frac1{n+1}\sum_{i=1}^{n+1}f_{n+1}(x;y_1, \ldots, y_n,x) \sigma(dy_1)\cdots \sigma(dy_n) \sigma(dx).
\end{align*}
This verifies \eqref{mSi} when $\sigma(\X)<\infty$, e.g., if $\X\subset \XX$ is compact; we note in passing that \eqref{emSia} is an explicit representation of either side of \eqref{mSi}.

We next let $\XX=\bigcup_m \X_m$ be a countable decomposition of $\XX$ 
into disjoint Borel sets with $\sigma(\X_m)<\infty$. For arbitrary process $f(x;\w)\ge 0$ we have
\begin{equation}\label{mSid}
\int_\XX f(x;\w) \w (dx)
= \sum_m \int_{\X_m} f(x;\w) \w (dx).
\end{equation}
For fixed $m$, we write $\w_* =\w\cap \X_m$, $\w^* =\w\setminus \X_m$, and denote
by $\E_*$ and $\E^*$ the expectation $\E$ when restricted to random variables  depending only on $\X_m$ and $\XX\setminus \X_m$, respectively.
By \eqref{eq:ind} and by \eqref{mSi} for $\X_m$, 
\begin{align*}
&\E \int_{\X_m} f(x;\w) \w (dx)
=\E^*\E_* \int_{\X_m} f(x;\w_* \cup \w^* ) \w_*  (dx)\\
&=\E^* \int_{\X_m} \E_* f(x;\w_* \cup \{x\}\cup \w^* ) \sigma(dx)
= \int_{\X_m}\E f(x;\w \cup \{x\}) \sigma(dx).
\end{align*}
This yields \eqref{mSi} in the general case, cf. \eqref{mSid}.

Needless to say, (\ref{mSi}) also holds for signed processes $f$ under the assumption of absolute integrability, because we can decompose both sides of (\ref{mSi}) according to $f=f_+-f_-$, where $f_+=\max \{f,0\}$ and $f_-=\max\{-f,0\}$.


\end{document}